\documentclass[11pt]{amsart}
\theoremstyle{plain}
\newtheorem{thm}{Theorem}[section]
\newtheorem{theorem}[thm]{Theorem}

\newtheorem{lemma}[thm]{Lemma}
\newtheorem{corollary}[thm]{Corollary}
\newtheorem{proposition}[thm]{Proposition}
\theoremstyle{definition}
\newtheorem{remark}[thm]{Remark}

\newtheorem{question}[thm]{Question}

\numberwithin{equation}{section}

\newcommand{\C}{{\mathbb C}}

\newcommand{\F}{{\mathbb F}}

\newcommand{\BP}{{\mathbb P}}
\newcommand{\Q}{{\mathbb Q}}
\newcommand{\R}{{\mathbb R}}

\newcommand{\Z}{{\mathbb Z}}

\title [Primitive birational automorphisms]{Simple abelian varieties and primitive automorphisms of null entropy of surfaces}

\author{Keiji Oguiso}
\address{Department of Mathematics, Osaka University, Toyonaka 560-0043, Osaka, Japan and Korea Institute for Advanced Study, Hoegiro 87, Seoul, 
133-722, Korea}
\email{oguiso@math.sci.osaka-u.ac.jp}

\thanks{The author is supported by JSPS Grant-in-Aid (S) No 25220701, JSPS Grant-in-Aid (S) No 22224001, JSPS Grant-in-Aid (B) No 22340009, and by KIAS Scholar Program.}
\dedicatory{Dedicated to Professor Tetsuji Shioda on the occasion of his
seventy-fifth birthday.}
\begin{document}

\maketitle

\begin{abstract} We characterize simple complex abelian varieties and simple abelian surfaces in terms of primitivity of translation automorphisms. Applying this together with a result due to Diller and Favre, we then classify all primitive birational automorphisms with trivial first dynamical degree of smooth projective surfaces over an algebraically closed field of any characteristic. 
\end{abstract}

\section{Introduction}

We shall work over a fixed algebraically closed field $k$ of characteristic $p \ge 0$. We are particularly interested in the cases where $k = \C$, $\overline{\Q}$, i.e., an algebraic closure of the prime field $\Q$ of characteristic $0$ and $\overline{\F}_p$, i.e., an algebraic closure of the prime field $\F_p$ of characteristic $p >0$. All technical terms will be explained in Section 2.

In their paper \cite{DF01}, Diller and Favre proved the following 
remarkable result on smooth projective surfaces:

\begin{theorem}\label{thm0}
Let $S$ be a smooth projective surface defined over $k$ and $g \in {\rm Bir}\, (S)$. 
Then: 

(1) The first dynamical degree $d_1(g)$ (see Section 2 for definition) satisfies that $d_1(g) \ge 1$ and $d_1(g)$ is a birational invariant in the sense 
that $d_1(g') = d_1(g)$ if $(S', g')$ is birationally conjugate to $(S, g)$, i.e., if $(S', g')$ is birational to $(S, g)$ as pairs. 

(2) If $d_1(g) = 1$, then either one of the following holds:

(i) there is a pair $(S', g')$ being birationally conjugate to $(S, g)$ such that $S'$ is smooth projective, $g' \in {\rm Aut}\, (S')$ and $(g')^n \in {\rm Aut}^0(S')$ for some $n >0$;

(ii) $g$ is not primitive, more precisely, there is a $g$-equivariant dominant rational map $S \cdots\to C$ over a smooth projective curve $C$ whose general fiber is isomorphic to either a possibly singular rational curve or a smooth  elliptic curve.
\end{theorem}
Here ${\rm Aut}^0(S)$ is the identity component of ${\rm Aut}\, (S)$. 
\begin{remark}\label{rem0}

(1) In \cite{DF01}, Diller and Favre proved Theorem \ref{thm0} when $k = \C$. However, their proof is valid even for $k$ with $p > 0$. In fact, there is only one place in \cite{DF01} where they used the theory of positive currents which is avaiable only for $k = \C$. It is \cite[Proposition 1.13]{DF01}. However, one can replace $H_{\R}^{1,1}(S)$ by ${\rm NS}\, (S) \otimes \R$ and redefine $\alpha \le \beta$ for $\alpha$, $\beta$ in ${\rm NS}\, (S) \otimes \R$ to be that $\beta -\alpha$ is pseudo-effective, rather than represented by a positive closed current. As not only the pull-back but also the pushforward preserve the nef cone and the pseudo-effective cone when $S$ is a smooth projective surface, their proof of the above mentioned theorem works over any $k$. See also \cite[Section 3]{Tr12}. If in addition $g \in {\rm Aut}\, (S)$, then $d_1(g) = 1$ if and only if $g$ is {\it of null entropy}. See \cite{ES13} for the notion of entropy in $p > 0$ and \cite{BC13} for the arithmetic nature of $d_1(g)$. 

(2) In \cite{DS05-1} (see also \cite{DS05-2}), Dinh and Sibony generalized Theorem \ref{thm0} (1) for the $l$-th dynamical degrees $d_l(g)$ of a birational automorphism $g$ of a smooth {\it complex} projective variety of any dimension. Unlike the case of surfaces, the role of currents in \cite{DS05-1} and \cite{DS05-2} seems really essential in dimension $\ge 3$. 

(3) It is well-known that $g$ is primitive if $d_1(g) > 1$ for a smooth {\it complex} projective surface $S$, hence for a smooth projective surface over $k$ with $p=0$. This is a straightforward consequence of the theory of relative dynamical degree due to Dinh, Nguy\^en and Truong \cite{DN11}, \cite{DNT11}, which again heavily depends on the theory of currents. However, Professors Cantat and Truong kindly informed that one will be able to confirm it also in algebraic setting at least for surfaces, either by using \cite{BC13} or by expanding \cite[Section 3]{Tr12}, the later of which is now under preparation (\cite{Tr14}).  
\end{remark}

In Theorem \ref{thm0} and Remark \ref{rem0}, it seems less known about the following:

\begin{question}\label{ques1} Let $S$ be a smooth projective surface defined over $k$ and $g \in {\rm Bir}\, (S)$. 

(1) Is there a pair $(S, g)$ such that $g$ is primitive and of $d_1(g) = 1$ (even when $k = \C$)? 

(2) If there is, can one classify all such pairs $(S, g)$ up to birational equivalence (even for $k = \C$)? 

\end{question}

The aim of this note is to answer to Question \ref{ques1} as a partial refinement of Theorem \ref{thm0} and a result of Hu, Keum and De-Qi Zhang \cite[Section 3]{HKZ14}:

\begin{theorem}\label{thm1}
Let $(S, g)$ be a pair consisting of a smooth projective surface defined over $k$ and a birational automorphism $g \in {\rm Bir}\, (S)$. Then: 

(1) If $g$ is primitive and $d_1(g) = 1$, then $(S, g)$ is birationally conjugate to $(A, t_P)$, where $A$ is a simple abelian surface and $t_P$ is the translation automorphism of $A$ corresponding to a non-torsion point $P \in A(k)$. In particular, $g$ is of infinite order. 

(2) Conversely, if $(S, g)$ is birationally conjugate to a pair $(A, t_P)$ of a simple abelian surface $A$ and a translation automorphism $t_P$ corresponding to a non-torsion point $P \in A(k)$, then $g$ is primitive and $d_1(g) = 1$.  
\end{theorem}

Here an abelian variety $A = (A, +, 0)$ is a complete group variety defined over $k$ (\cite{Mu70}). Then $A$ is always projective as a variety and abelian as a group. When $k = \C$, the underlying complex manifold $A(\C)$ is isomorphic to a projective complex torus $\C^m/L$, where $L \subset \C^m$ is a discrete free $\Z$-submodule of rank $2m$. $A$ is {\it simple} if $A$ has no irreducible closed algebraic subgroup other than $\{0\}$ and $A$. We call an abelian variety of dimension $2$ (resp. $1$) an abelian surface (resp. an elliptic curve). 

As a special case of Theorem \ref{thm1}, we obtain:

\begin{corollary}\label{cor1}
(1) Let $k = \overline{\Q}$. Then, there is a pair $(S, g)$ consisting of a smooth projective surface and a birational automorphism $g \in {\rm Bir}\, (S)$, defined over $k$, such that $d_1(g) = 1$ and $g$ is primitive. The same is true also over an algebraic closure $k = \overline{{\F_q}(C)}$ of the rational function field $\F_q(C)$ of a smooth projective curve $C$ defined over a finite field $\F_q$ of characteristic $p > 0$.

(2) Let $k = \overline{\F}_p$.
Then, there is no pair $(S, g)$ consisting of a smooth projective surface and a birational automorphism $g \in {\rm Bir}\, (S)$, defined over $k$, such that $d_1(g) = 1$ and $g$ is primitive.  
\end{corollary}

In Corollary \ref{cor1}, $\overline{\Q}$, $\overline{{\F_q}(C)}$ and $\overline{\F}_p$ are countable. We also note that there are a simple abelian surface defined over $\overline{\Q}$ by Mori \cite{Mo77} and a simple abelian surface defined over any algebraically closed field by Howe and Zhu \cite{HZ02}. 

Besides Theorem \ref{thm0}, our proof is based on Theorem \ref{prop2}, which is a slight reformulation of a nice observation originally due to De-Qi Zhang \cite{Zh09} (see also \cite{NZ09}) and Theorem \ref{prop2-1} on a characterization of simple abelian varieties in terms of the primitivity of translations (see also Corollary \ref{cor2-1}). I believe that these two theorems are of their own interest. 

\begin{remark}\label{remrem1}
Theorem \ref{prop2-1} and Proposition \ref{lem2-2} seem also closely related to a result announced by D. Ghioca and T. Scanlon \cite{GS14} quite recently.
\end{remark}

\par
\medskip

{\bf Acknowledgements.} I would like to express my sincere thanks to Professors Eric Bedford, Fabrizio Catanese, Fabio Perroni and Thomas Peternell for inspiring discussions at Trieste September 2014, from which this note was grown up. I would like to express my best thanks to Professors Serge Cantat, Truyen Truong, Xun Yu, De-Qi Zhang, and especially to Professor H\'el\`ene Esnault, for many valuable discussions, comments and improvements. I would like to express my thanks to Doctor Yoshinosuke Hirakawa for his careful reading of the preliminary version and useful comments. 

\section{Notations.}

In this section, we fix notations we will use. 

As mentioned in Introduction, we work over a fixed algebraically closed field $k$ of characteristic $p \ge 0$. 

Let $M$ be a normal projective variety of dimension $m$ 
and $f \in {\rm Bir}\, (M)$, i.e., $f$ is a birational automorphism $f$ of 
$M$. 

Let $\varphi : M \cdots \to B$ be a dominant rational map to a normal projective variety $B$. Then there is the maximum Zariski open subset ${\rm dom}\, (\varphi)$ such that $\varphi | {\rm dom} (\varphi)$ is a morphism. We denote the indeterminacy locus of $\varphi$ by 
$$I(\varphi) := M \setminus {\rm dom}\, (\varphi)\,\, .$$ 
Then, as $M$ is normal, $I(\varphi)$ is a Zariski closed subset of $M$ of codimension $\ge 2$. 

Let $\Gamma_{\varphi} \subset M \times B$ be the graph of $\varphi$, that is, the Zariski closure of the graph of the morphism
$$\varphi | X \setminus I(\varphi) : X \setminus I(\varphi) \to B$$
in $M \times B$, and $p_1 : \Gamma_{\varphi} \to M$, $p_2 : \Gamma_{\varphi} \to B$ be the natural projections. For Zariski closed subsets $S \subset M$ and $R \subset B$, we define with reduced scheme structure
$$\varphi(S) := p_2(p_1^{-1}(S))\,\, ,\,\, \varphi^{-1}(R) := p_1(p_2^{-1}(R))\,\, .$$
Note that $\varphi(S)$ is a Zariski closed subset of $B$ and $\varphi^{-1}(R)$ is a Zariski closed subset of $M$, as both $p_1$ and $p_2$ are proper. 
The fiber $F_b$ over a point $b \in B(k)$ with reduced scheme structure is defined by
$$F_b := \varphi^{-1}(b)\,\, .$$

A dominant rational map $\varphi : M \cdots\to B$ is called a {\it rational fibration} if $F_b$ is connected for general $b \in B$ and hence for all $b \in B$ (the Stein factorization theorem). If in addition $p = 0$, then general fiber $F_b$ is irreducible (Bertini's theorem). A rational fibration $\varphi$ is {\it non-trivial} if $0 <\dim\, B < \dim\, M$. 

A birational automorphism $f \in {\rm Bir}\, (M)$ is called {\it imprimitive} if there are a non-trivial rational fibration $\varphi : M \cdots \to B$ and a rational map $f_B : B \cdots \to B$, necessarily $f_B \in {\rm Bir}\, (B)$, such that 
$$\varphi \circ f = f_B \circ \varphi$$
as rational maps from $M$ to $B$. A birational automorphism that is not imprimitive is {\it primitive}. This definition is consistent to the one defined by De-Qi Zhang in complex analytic case (\cite{Zh09}), by Lemma \ref{lem0}. We do {\it not} assume that $f$ is of infinite order.

Let $H$ be an ample Cartier divisor on a normal projective variety $M$. The $l$-th {\it dynamical degree} 
$d_l(f)$ of $f \in {\rm Bir}\, (M)$ is defined by 
$$d_{l}(f) := \limsup_{n \to \infty} (\delta_{l}(f^n))^{\frac{1}{n}}\,\, .$$ 
Here $\delta_{l}(f^n) $ is the intersection number (well-)defined by:
$$\delta_{l}(f^n) := ((\pi_{2, n}^*H)^l.(\pi_{1, n}^*H)^{m-l})_{\Gamma_{f^n}} 
= ((\pi_{1,n})_*((\pi_{2, n}^*H)^l).H^{m-l})_{M}$$
where $\Gamma_{f^n}\subset M \times M$ is the graph of $f^n$ and 
$$\pi_{i, n} : {\Gamma}_{f^n} \to M$$ 
($i=1$, $2$) are the natural projections, $\pi_{i, n}^*$ is the natural pullback map as Cartier divisors, and $(\pi_{i, n})_*$ is the natural pushforward map as cycles. $d_i(f)$ does not depend on the choice of $H$, as for another ample Cartier divisor $L$, there are positive integers $n_1$, $n_2$ and ample divisors $H'$ and $L'$ such that $n_1H = L + L'$ and $n_2L = H + H'$ in ${\rm Pic}\, (M)$. 

Let $M_i$ ($i =1$, $2$) be normal projective varieties and $f_i \in {\rm Bir}\, (M_i)$. We call $(M_1, f_1)$ is {\it birationally conjugate} to $(M_2, f_2)$ if there is a birational map $\mu : M_1 \cdots\to M_2$ such that 
$$\mu \circ f_1 = f_2 \circ \mu$$
as rational maps from $M_1$ to $M_2$. 

We say that a general closed point $P$ of a variety $V$ over a algebraically closed field $k$ satisfies a condition (C) if there is a non-empty Zariski open set $U \subset V$ such that (C) holds for all $P \in U(k)$.  

We call a variety defined over $\C$ a complex variety. 

\section{Some generalities of primitive birational automorphisms.}

We work over an algebraically closed field $k$. We begin with the following two more or less trivial lemmas. 

\begin{lemma}\label{lem0}
Let $M$ be a complex projective algebraic variety and $\varphi : M(\C) \cdots\to B$ a dominant meromorphic map onto a compact complex analytic space $B$. 
Then $B$ 
is bimeromorphic to the underlying analytic space of a complex projective variety $B'$ and the induced map $\psi : M \cdots\to B'$ is a dominant rational map. 
\end{lemma}

\begin{proof} By \cite[Chap I, Sect. 3, Cor. 3.10]{Ue75}, $B$ is Moishezon. Hence Chow's theorem implies the result.
\end{proof}

We denote by $|L|$ the complete linear system associated to a line bundle $L$. 
We say that $f \in {\rm Bir}\, (M)$ is isomorphic in codimension one if there are Zariski closed subsets $V, W \subset M$ of codimension $\ge 2$, possibly empty, such that 
$$f | M \setminus V : M \setminus V \to M \setminus W$$ 
is an isomorphism. 
\begin{lemma}\label{lem1}
Let $M$ be a normal projective variety defined over $k$, and $f \in {\rm Bir}\, (M)$. Assume that $m := \dim\, M \ge 2$.

(1) Assume that there is $L \in {\rm Pic}\, (M)$ such that $\dim\, |L| \ge 1$ 
and $f^{*} : |L| \to |L|$ is a well-defined projective automorphism. Then $f$ is not primitive.

(2) Assume that either $f \in {\rm Aut}\, (M)$ or $M$ is factorial and $f$ is isomorphic in codimension one. Assume further that (now well-defined) $f^* \in {\rm Aut}\, ({\rm Pic}\, (M))$ is of finite order. Then $f$ is not primitive. 

(3) Assume that $f$ is of finite order. Then $f$ is not primitive. 
\end{lemma}

\begin{proof} 
As $|L| \simeq {\mathbb P}^l$ with $l \ge 1$ and $f^*$ is projective linear, one finds an $f^*$-stable projective line in $|L|$, i.e., an $f^*$-stable pencil $\Lambda \subset |L|$. Indeed, one may consider the matrix representation of $f^*$ and its Jordan canonical form. Let  
$$\Phi_{\Lambda} : M \cdots\to \Lambda^* \simeq \BP^1$$
be the rational map associated to $\Lambda$. Then $\Phi_{\Lambda}$ is dominant 
and $f$-equivariant under the natural action of $f$ on $\Lambda^*$. Let $\Gamma$ be the normalization of the graph of $\Phi_{\Lambda}$ and $p : \Gamma \to \BP^1$ be the induced morphism. 
Then $p$ is surjective and $f$-equivariant under the natural action of $f$ on $\Gamma$, say $f_{\Gamma}$. Let
$$\varphi : \Gamma \to C$$ 
be the Stein factorization of $p : \Gamma \to \BP^1$. Then $\varphi$ is an $f$-equivariant non-trivial fibration. Here $\dim\, \Gamma = \dim\, M \ge 2$. Thus $f_{\Gamma}$ is not primitive. $(M, f)$ is birationally conjugate to $(\Gamma, f_{\Gamma})$ by construction. Hence $f$ is not primitive as well. This proves (1).

Let us prove (2). Note that $f^* \in {\rm Aut}\,({\rm Pic}\, (M))$, which is well-defined by the assumption. Moreover, $(f^l)^* = (f^*)^l$ on ${\rm Pic}\, (M)$ as $f$ is isomorphic in codimension one. 

Let $n$ be the order of $f^*$. Choose a very ample Cartier divisor $H$. Then the complete linear system 
$$|L| := |\sum_{j=0}^{n-1}(f^{j})^{*}H|$$ 
is $f^*$-stable by $(f^l)^* = (f^*)^l$. We can then apply (1) for $|L|$ to conclude (2).

Let us prove (3). Let $d$ be the order of $f$. Let us consider the natural action $f^*$ of $f$ on the rational function field $Q(M)$ of $M$. Then $f^*$  is of finite order $d$. Note that $f^* = id$ on $k$, the subfield of $Q(M)$ consisting of constant functions. The field $Q(M)$ is finitely generated over $k$. The invariant field $Q(M)^{f^{*}}$ is also a finitely generated field over $k$ with the same transcendental degree as $Q(M)$. Indeed, if we write $Q(M) = k(t_i | 1 \le i \le e)$ and define $s_{j}(x_i | 1 \le i \le d)$ to be the elementary symmetric polynomial of degree $j$ of $d$ variables, then 
$$S(M) := k(s_j(t_i, f^*(t_i), \cdots , (f^*)^{d-1}(t_i)) | 1 \le i \le e, 1 \le j \le d) \subset Q(M)^{f^*}\,\, ,$$
and 
$$Q(M)^{f^*} \subset Q(M) = k(t_i | 1 \le i \le e)\,\, .$$
Each $t_i$ ($1 \le i \le e$) is a root of a polynomial of degree $d$ in $S(M)[t]$ by definition, and $Q(M)$ is generated by $t_i$ ($1 \le i \le e)$ also over $S(M)$. Hence 
$$[Q(M) : S(M)] \le d^e$$
 by the field theory. We have
$$[Q(M) : S(M)] = [Q(M):Q(M)^{f^*}] \cdot [Q(M)^{f^*} : S(M)]$$
again by the field theory. Thus, both $[Q(M)^{f^*} : S(M)]$ and $[Q(M) : Q(M)^{f^*}]$ are finite. Since $S(M)$ is finitely generated over $k$, it follows that $Q(M)^{f^*}$ is also finitely generated over $k$ and has the same transcendental degree as $Q(M)$. 

Thus, there is a normal projective variety $V$ whose generic point in the sense of scheme is isomorphic to ${\rm Spec}\, Q(M)^{f^*}$. Let $M'$ be the normalization of $V$ in 
$Q(M)$. Then $M'$ is also normal projective and birational to $M$. By the construction and the uniqueness of the normalization in $Q(M)$, the birational conjugate $f'$ of $f$ on $M'$ is now a {\it biregular} automorphism of $M'$, of finite order. Thus $f'$ is not primitive on $M'$ by (2). As $(M, f)$ is birationally conjugate to $(M', f')$, it follows that $f$ is not primitive as well. 
\end{proof}

The next theorem is a slight modification of a nice observation due to De-Qi Zhang (\cite{Zh09}, see also \cite{NZ09} for a generalization and \cite{Og14} for a general survey):

\begin{theorem}\label{prop2} Let $M$ be a smooth projective variety of dimension $m$ defined over $k$ of characteristic $p = 0$. Assume that the following two statements hold for $M$:

(1) If $\kappa(M) = 0$ and $q(M) := h^1(M, {\mathcal O}_M) = 0$, then $M$ is birational to a generalized minimal Calabi-Yau variety, i.e., a normal projective varietry $M'$ with only $\Q$-factorial terminal sinularities with ${\mathcal O}_{M}(kK_{M'}) \simeq {\mathcal O}_{M'}$ for some $m >0$ 
and $h^1({\mathcal O}_{M'}) = 0$; and 

(2) If $\kappa(M) = -\infty$, then $M$ is uniruled. 

Then, if $M$ has a primitive birational automorphism $f$, then either one of the following holds:

(RC) $M$ is a rationally connected manifold, i.e., a smooth projective manifold whose any two general closed points are connected by a rational curve;

(GCY) $M$ is birational to a generalized minimal Calabi-Yau variety; or

(T) $M$ is birational to an abelian variety. 

\end{theorem}

\begin{remark}\label{rem1} 

(1) Theorem \ref{prop2} is unconditional if $m \le 3$ by the minimal model theory and abundance theorem for projective threefolds due to Kawamata, Miyaoka, Mori and Reid (\cite{Mo88}, \cite{Ka92}, see also \cite{KM98} and references therein). 

(2) Moreover, the proof below shows that if one replaces (GCY) by a K3 surface or an Enriques surface (three types when $p=2$), then the same conclusion as Theorem \ref{prop2} holds unconditionally for birational automorphisms of smooth projecive surfaces defined over $k$ of any characteristic $p \ge 0$. This is because we have the classification of surfaces due to Bombieri and Mumford \cite{BM77}, as we shall indicate in the proof below. 
\end{remark}

\begin{proof} 

Note that the pull-back of the regular $l$-th pluricanonical form under the birational map $f \in {\rm Bir}\, (M)$ is again regular, as $I(f)$ is of codimension $\ge 2$ and $M$ is smooth. Thus $f^*$ induces a well-defined linear injective selfmap of $H^0(M, lK_M)$, hence induces a regular automorphism of the projective space $|lK_M|^*$. 

(i) If $\dim |lK_M| \ge 2$ for some $l >0$, then the result follows from Lemma \ref{lem1} that $f$ is not primitive. Thus $f$ is not primitive if $1 \le \kappa (M)$. (Here we do not use $p = 0$.)

(ii) Consider the case where $\kappa (M) = 0$. Note that $M$, the albanese variety ${\rm Alb}\, (M)$ of $M$ and the albanese morphism ${\rm alb}_M$ and the birational automorphism $f$ are all defined over a finitely generated subfield of $k_0$ over the prime field $\Q$ of $k$. Embed $k_0$ into $\C$. For our purpose, we may assume without loss of generality that all are defined over ${\mathbb C}$ by taking the fiber product with ${\rm Spec}\, \C$ over ${\rm Spec}\, k_0$. If $q(M) > 0$, then we may describe the albanese morphism as
$${\rm alb}_M : M \to {\rm Alb}\, (M) := H^0(M, \Omega_M^1)^{*}/H_1(M, \Z)\,\, .$$
It is classical that ${\rm Alb}\, (M)$ is an abelian variety. Since $\kappa (M) = 0$ and $M$ is projective, a fundamental theorem due to Kawamata \cite{Ka81} says that ${\rm alb}_M$ is surjective with connected fibers. In particular, $q(M) \le m$. For the same reason as before, the action $f$ descends to the biregular action on ${\rm Alb}\, (M)$ equivariantly with respect to the albanese map. Hence either $q(M) = 0$ or $q(M) = m$. In the second case, $M$ is birational to ${\rm Alb}\, (M)$ via ${\rm alb}_M$. In the first case, by the assumption made, $M$ is birational to a generalized Calabi-Yau variety. (Let $M$ be of dimension $2$, $\kappa\, (M) = 0$ and $k$ be of any characteristic. Then $M$ is birational to either an abelian surface ($q(M) = 2$), (quasi-) bielliptic surfaces ($q(M) = 1$), an Enriques surface or a K3 surface ($q(M) = 0$) by the classification of surfaces \cite{BM77}. So, the result follows, as the albanese map of an (quasi-) bielliptic surface is birational invariant.)

(iii) It remains to treat the case $\kappa (M) = -\infty$. Then $M$ is uniruled by the assumption made above. Then the maximally rationally connected fibration $\pi : M \cdots \to V$ is a ${\rm Bir}\, (M)$-equivariant rational fibration to a lower dimensional $V$ (see eg. \cite[Cahp IV, Sect. 3]{Ko96}). Hence $V$ is a point, i.e., $M$ is rationally connected, if $f$ is primitive.  (Let $M$ be of dimension $2$, $\kappa\, (M) = -\infty$ and $k$ be of any characteristic. Then $M$ is  a rational surface ($q(M) = 0$), 
or blow-downs to a minimal ruled surface over a smooth curve $C$ of genus $q(M) > 0$, by the classification of surfaces \cite{BM77}. The second case, the induced morphism $M \to C$ is equivariant under ${\rm Bir}\, (M)$.)
\end{proof} 

\begin{corollary}\label{cor3} Let $S$ be a smooth projective surface defined over an algebraically closed field $k$. Let $g$ be a birational automorphism of 
$S$. If $g$ is primitive, then $S$ is birational to either:

(i) $\BP^2$;

(ii) a K3 surface or an Enriques surface; or

(iii) an abelian surface. 

\end{corollary}

\begin{proof} This follows from Theorem \ref{prop2} and Remark \ref{rem1}. 
Compare also with a pioneering work due to Cantat \cite{Ca99}. 
\end{proof} 

\begin{remark}\label{rem2}
In each case (i), (ii), (iii), there are {\it complex} projective surface $S$ and $g \in {\rm Aut}\, (S)$ such that $g$ is primitive, more strongly, of $d_1(g) > 1$. See, for instance, \cite{Mc07}, \cite{BK09}, \cite{BK12} for rational $S$, \cite{Mc11-2} for K3 surfaces, \cite{CO11} for Enriques surfaces and \cite{Mc11-1}, \cite{Re12} for abelian surfaces. See also \cite{Ca99} and \cite{Og14}. 
\end{remark}

The algebraic dimension of a compact complex manifold  $M$ is the transcendental degree of the meromorphic function field of $M$ (\cite[ChapI, Sect.3]{Ue75}). We denote it by $a(M)$. Then $a(M) \le \dim\, M$ and $M$ is algebraic if and only if $a(M) = \dim\, M$. 

\begin{corollary}\label{cor2} Let $S$ be a smooth compact, non-algebraic K\"ahler surface. If $S$ has a primitive bimeromorphic automorphism (not necessarily of infinite order) $g$, then $S$ is bimeromorphic to either:

(i) a K3 surface of algebraic dimension $0$; or

(ii) a $2$-dimensional complex torus of algebraic dimension $0$. 

\end{corollary}

\begin{proof} As $S$ is not algebraic, $a(S) \not = 2$. As the algebraic reduction map (\cite[ChapI, Sect.3]{Ue75}) is a $g$-equivariant dominant meromorphic fibration, $a(S) \not= 1$. Hence $a(S) = 0$. 

The result now follows from the classification of compact complex K\"ahler surfaces and the proof of Theorem \ref{prop2}, in which Bombieri-Mumford's classification of algebraic surfaces is replaced by Kodaira's classification of compact complex surfaces (see eg. \cite[Chap VI]{BHPV04}). 
\end{proof} 

\begin{remark}\label{rem3}
Note that $S$ in (i), (ii) does not admit any dominant meromorphic map to a compact Riemann surface, as any compact Riemann surface is algebraic. In particular, all bimeromorphic automorphisms of $S$ are primitive in these cases. Again, in each case (i), (ii), there are a {\it complex} surface $S$ and $g \in {\rm Aut}\, (S)$ such that $g$ is primitive, {\it much more strongly}, of $d_1(g) > 1$. See \cite{Mc02-2}, \cite{Mc07} for very impressive such examples.
\end{remark}

\section{A dynamical characterization of simple abelian varieties.}

\begin{lemma}\label{lem2-1}
Let $T = (T, 0)$ be a simple abelian variety defined over an algebraically closed field $k$ with the unit element $0 \in T(k)$. Let $P \in T(k)$ be a non-torsion point and $t_P$ be the translation $t_P \in {\rm Aut}\, (T)$ associated to $P$. Then the Zariski closure of the orbit $\langle t_P \rangle \cdot Q$ is $T$ for each $Q \in T(k)$. 

The same is true for a simple complex torus $T = (T, 0)$, which is not necessarily algebraic. 
\end{lemma}

\begin{proof} 
The Zariski closure $V_P$ of $\langle P \rangle \subset T$ is a projective algebraic subgroup of $T$ of positive dimension (as $P$ is not a torsion point). Hence the identity component $V_P^0$ of $V_P$ is an abelian subvariety of $T$. Here we used the fact that any irreducible algebraic subgroup of an abelian variety is an abelian subvariety. Hence $V_P^0 = T$, as $T$ is simple. Thus, $V_P = T$. 

Note that $V_P$ is the same as the Zariski closure of the orbit $\langle t_P \rangle \cdot 0$ of $0$ in $T$. As the addition by $Q$ is an automorphism of the variety $T$, the Zariski closure of the orbit $\langle t_P \rangle \cdot Q$ of $Q$ in $T$ is the same as the translation by $Q$ of $V_P = T$, the Zariski closure of the orbit $\langle t_P \rangle \cdot 0$. This implies the result.  

Proof for a simple complex torus is essentially the same as above. 
\end{proof}

We call an abelian variety or a complex torus {\it non-simple} if it is not simple. 

\begin{proposition}\label{lem2-2}
Let $T$ be a non-simple abelian variety, necessarily, of dimension $m \ge 2$. Let $P \in T(k)$ be any point. Then the translation $t_P \in {\rm Aut}\, (T)$ associated to $P$ is not primitive. 

The same holds for a non-simple complex torus. 
\end{proposition}

\begin{proof} 
By assumption, one can find an abelian subvariety of $T$ of positive dimension other than $T$ itself. We denote it by $E$. Then 
$$P + E = E + P$$ 
in $T$. Thus the natural quotient morphism $T \to T/E$ is a $t_P$-equivariant non-trivial fibration onto $T/E$, which is an abelian variety 
of dimension $0 < \dim\, T/E < \dim\, T$. Hence $t_P$ is not primitive.

Proof for a non-simple complex torus is essentially the same as above. 
\end{proof}

\begin{theorem}\label{prop2-1}
Let $T$ be a simple complex abelian variety of any dimension $m \ge 2$ or a simple abelian surface defined over any algebraically closed field $k$ of any characteristic. Let $P$ be a closed point of $T$. Then the translation $t_P \in {\rm Aut}\, (T)$ is primitive if and only if $P$ is not a torsion point. Moreover $d_1(t_P)$ is $1$.
\end{theorem}

\begin{proof} 
The last statement directly follows from the definition of the first dynamical degree. {\it Only if part} follows from  Lemma \ref{lem1} (3). 

We show {\it if part}. Assume to the contrary that there is a $t_P$-equivariant rational fibration 
$$\varphi : T \cdots\to B$$ 
with $0 < \dim B < \dim T$ to a normal projective variety $B$. Let $t_B$ be the birational automorphism of $B$ induced by $t_P$. We may assume that $B$ is smooth by a resolution of singularities in both cases. 

Let $\nu : V \to T$ be a Hironaka's resolution of the indeterminacy $I(\varphi)$, i.e., a Hironaka's resolution of the graph $\Gamma_{\varphi}$ of $\varphi$, in both cases. 

Then $\nu^{-1}(I(\varphi))$ is a Zariski closed subset of $V$ of pure codimension one and 
$$V \setminus \nu^{-1}(I(\varphi)) \simeq T \setminus I(\varphi)$$ 
under $\nu$. Let 
$$\psi := \varphi \circ \nu : V \to B$$
 be the induced morphism. $\psi$ is proper and surjective by the construction. Therefore, $\psi(\nu^{-1}(I(\varphi))$ is a closed algebraic subset of $B$. 

Let $b$ be a closed point of $B$. Recall that the fiber $F_b$ of $\varphi : T \cdots\to B$ over $b$ is defined by 
$$F_b := \nu(\psi^{-1}(b)) \subset T\,\, .$$ 
This is a Zariski closed subset of $T$. 

We call $\varphi$ is {\it almost regular} if $\psi(\nu^{-1}(I(\varphi))) \not= B$, i.e., 
if $\psi(\nu^{-1}(I(\varphi))$ is of codimension $\ge 1$ in $B$. 

Let $b$ be a general closed point of $B$. If $\varphi$ is almost regular, then 
$0 < \dim\, F_b < \dim\, T$ and the dualizing sheaf of $F_b$ is trivial by the adjunction formula. If $k = \C$, then $F_b$ is a translation of an abelian subvariety of $T$. This is due to the classification of subvariety of complex tori \cite[Chap. IV, Sect 10]{Ue75}. However, this is impossible, as $T$ is simple. If $m=2$ and $k$ is any field, then $F_b$ is a smooth elliptic curve, as $T$ has no rational curve. However, this is again impossible, as $T$ is simple. 

From now, we may and will assume that $\varphi$ is {\it not} almost regular. 

Let $E = \cup_{l=1}^{n} E_l$ be the union of all irreducible components of $\nu^{-1}(I(\varphi))$ such that $\psi(E_l) = B$ and $\nu(E_l)$ is an irreducible component of $I(\varphi)$. By the definition of $I(\varphi)$ and by the fact that $\varphi$ is not almost regular, $E$ is non-empty. 

First consider the case where $m=2$ and $k$ is any algebraically closed field. Then $\nu$ is a finite composition of blowing-ups of closed points. Thus $E$ is a union of rational curves. Hence $B$ is $\BP^1$ by L\"uroth's theorem. Then $t_B \in {\rm PGL}(2, k)$ as
$${\rm Bir}\, (\BP^1) = {\rm Aut}\, (\BP^1) = {\rm PGL}\, (2, k)\,\, .$$
Thus $t_B(b) = b$ for some closed point $b$ of $B = \BP^1$. Indeed, the closed point corresponding to an eigenvector of the matrix $\tilde{g}$ representing $g$ is such a point. Thus $t_P(Q) \in F_b(k)$ for any 
$$Q \in F_b^0 := F_b(k) \setminus (I(\varphi) 
\cup t_P^{-1}(I(\varphi)))\,\, .$$
Here $I(\varphi) \cup t_P^{-1}(I(\varphi))$ is a finite set of closed points. Thus $F_b^0$ is a Zariski dense subset of $F_b$. Hence 
$$t_P(F_b) \subset F_b$$
 as both $F_b$ and $F_b$ are Zariski closed and $t_P$ is an automorphism, in particular, a closed map in the Zariski topology. The same is true for $t_P^{-1}$. Hence $$\langle t_P \rangle \cdot Q \subset F_b(k)$$ 
for any $Q \in F_b(k)$. Again, as $F_b$ is Zariski closed in $T$, the Zariski closure of $\langle t_P \rangle \cdot Q$ in $T$ is then a subset of $F_b$. As $F_b \not= T$, this contradicts Lemma \ref{lem2-1}. Hence $t_P$ is primitive if $m=2$.

Let us consider the case where $m \ge 3$ and $k = \C$. Recall that $\C$ is uncountable and of characteristic $0$. Consider the set 
$$B^0 = B(\C) \setminus \cup_{n \in \Z} t_B^n(I(t_B) \cup I(t_B^{-1}) \cup R \cup J)\,\, .$$
Here $R$ is the subset of $B(\C)$ consisting of points $b \in B(\C)$ such that $F_b$ is not irreducible, and $J$ is the subset of $B(\C) \setminus R$ consisting of points $b$ such that $F_b \subset I(\varphi)$. Here we note that 
$$I(t_B^m) \subset \cup_{n \in \Z} t_B^n(I(t_B) \cup I(t_B^{-1}))$$ 
for each $m \in \Z$ and $F_b$ is irreducible if $b \in B(\C) \setminus R$. Then the Zariski closure of $I(t_B) \cup R \cup J$ in $B$ is of codimension $\ge 1$. Hence so is for $t_B^n(I(t_B) \cup R \cup J)$ for each $n \in \Z$, as $t_P^{\pm}$ are automorphisms. Since $\C$ is uncountable and $\Z$ is countable, it follows that $B^0$ 
is a Zariski dense subset of $B$. Also, by the definition, the map
$$t_B^n |B^0 : B^0 \to B^0\, (\subset B)$$
 is set-theoretically well-defined for all $n \in {\mathbb Z}$. As all $t_P^n$, hence $t_B^n$, are birational, $B^0$ contains a complement of a countable union of codimension $\ge 1$ Zariski closed subsets. Again, as $\C$ is uncountable, there is then a Zariski dense subset $B^1 \subset B^0$ such that 
$$t_B^n |B^1 : B^1 \to B^0$$ 
is injective for all $n \in {\mathbb Z}$. Choose the largest such $B^1$. 

If $F_b \cap F_c = \emptyset$ for any $b \not= c \in B^1$, then $\varphi$ is everywhere defined over $B^1$. Hence $\psi(\nu^{-1}(I(\varphi)))$ is in $B \setminus B^1$, a contradiction to the assumption that $\psi(E) = B$. 

From now, we may and will assume that there are $b \not= c \in B^1$ such that $F_b \cap F_c$ contains a closed point, say $Q$. 

Then $Q \in I(\varphi)(\C)$ by the definition of $I(\varphi)$. Recall that $t_B \circ \varphi = \varphi \circ t_P$ as rational maps, $t_P^n$ is regular and $t_B^n : B_1 \to B_0$ is everywhere defined. Then, for each fixed $n$, the point $t_P^n(Q)$ is a well-defined point of $t_P^n(F_b)(\C) \cap t_P^n(F_c)(\C)$. We have
$$t_P^n(R) \in F_{t_B^n(b)}(\C)$$ 
for any points $R$ such that 
$$R \in F_b(\C) \setminus (I(\varphi) \cup t_P^{-n}(I(\varphi))\,\, .$$
As $b \in B^0$, such points $R$ form a Zariski dense subset of $F_b$. 
Hence 
$$t_P^n(S) \in F_{t_B^n(b)}(\C)$$
for {\it every} point $S \in F_b(\C)$. This is because $F_b$, $F_{t_B^n(b)}$ are both Zariski closed and $t_P^n$ is an automorphism, in particualr, a closed map in Zariski topology.  Thus
$$t_P^n(F_b) \subset F_{t_B^n(b)}\,\, .$$
For the same reason, we have
$$t_P^n(F_c) \subset F_{t_B^n(c)}\,\, .$$
Hence 
$$t_P^n(Q) \in F_{t_B^n(b)}(\C) \cap F_{t_B^n(c)}(\C)\,\, .$$
By the definition of $B^1$, $F_{t_B^n(b)} \not= F_{t_B^n(c)}$. 
Hence  
$t_{P}^n(Q) \in I(\varphi)(\C)$ for each fixed $n \in \Z$. Since $I(\varphi)$ is independent of $n$, it follows that $t_{P}^n(Q) \in I(\varphi)(\C)$ for all 
$n \in \Z$, 
that is,
$$\langle t_P \rangle \cdot Q \subset I(\varphi)(\C)\,\, .$$
As $I(\varphi)$ is Zariski closed in $T$, the Zariski closure of $\langle t_P \rangle \cdot Q$ is then a subset of $I(\varphi)$. As $I(\varphi) \not= T$, 
this contradicts Lemma \ref{lem2-1}. This completes the proof of Theorem \ref{thm1}.
\end{proof}

The following remark is kindly pointed out by Doctor Xun Yu:
\begin{remark}\label{xu} The proof here implies the following:

{\it Let $X$ be a smooth complex variety and $f \in {\rm Aut}\, (X)$ be an automorphism of $X$ such that for every $x \in X(\C)$, the orbit $\langle f \rangle \cdot x$ is Zariski dense in $X$. Then, either $f$ is primitive or $X$ admits a $f$-equivariant almost regular fibration $X \cdots\to B$.}

It might be interesting to find an explicit example of the later alternative. 
\end{remark}

\begin{corollary}\label{cor2-1}
Let $T$ be an abelian variety defined over $\C$ of positive dimension. Let $P \in T(\C)$ be a non-torsion point. Then $T$ is simple if and only if the translation $t_P \in {\rm Aut}\, (T)$ is primitive.
\end{corollary}

\begin{proof} This is a direct consequence of Proposition \ref{lem2-2} and Theorem \ref{prop2-1}. 
\end{proof}

Let $T$ be a complex abelian variety of positive dimension. We note that there is a non-torsion point $P \in T(\C)$. This is because the set of torsion points of $T(\C)$ is countable, while $T(\C)$ is uncountable. 

Let $\overline{\Q}$ be an algebraic closure of the field $\Q$ of rational numbers, i.e., an algebraic closure of the prime field of characteristic $0$. Then $\overline{\Q}$ is algebraically closed field but countable. However:

\begin{corollary}\label{qu2-1}
Let $T$ be an $m$-dimensional simple abelian variety defined over $\overline{\Q}$ such that $m \ge 2$. Let $P \in T(\overline{\Q})$ be a non-torsion point. Then the translation $t_P \in {\rm Aut}\, (T)$ is primitive.
\end{corollary}

\begin{proof} Assume to the contrary that we have a $t_P$ equivariant non-trivial rational fibration $T \cdots\to B$. We may assume that $\overline{\Q} \subset \C$. Then, by taking the fiber product $* \times_{{\rm Spec}\, \overline{\Q}} {\rm Spec}\, \C$, we have a $t_P$-equivariant non-trivial rational fibration 
$T_{\C} \cdots\to B_{\C}$ from a complex abelain variety $T_{\C}$, a contradiction to Theorem \ref{prop2-1}. We note that the natural morphism ${\rm End}(T) 
\to {\rm End}\, (T_{\C})$ is an isomorphism, as $\overline{\Q}$ is algebraically closed, so that $T_{\C}$ is simple as well over $\C$. 
\end{proof}

Next consider the case where $k = \overline{\F_q(C)}$, an algebraic closure of the function field of a smooth projective curve $C$ defined over a finite field $\F_q$ of characteristic $p > 0$. 

\begin{corollary}\label{qu2-4}
Let $T$ be a simple abelian surface defined over $\overline{\F_q(C)}$. Let $P \in T(\overline{\F_q(C)})$ be a non-torsion point. Then the translation $t_P \in {\rm Aut}\, (T)$ is primitive.
\end{corollary}

\begin{proof} As $T$ is a simple abelian surface over $\overline{\F_q(C)}$, this is a special case of Theorem \ref{prop2-1}.
\end{proof}

We note that Corollaries \ref{qu2-1}, \ref{qu2-4} are non-empty: Firstly, by \cite{Mo77} and \cite{HZ02}, there is an $m$-dimensional simple abelian variety defined over any algebraically closed field for any positive integer $m$. Secondly, there is certainly a non-torsion point by the following observation due to G. Frey and M. Jarden \cite{FJ72}: a more elementary proof for $\overline{\Q}$ is given by J. P. Serre, both of which Professor H\'el\`ene Esnault kindly informed me after sending her my preliminary note. For $\overline{\Q}$, this is now also an immediate corollary of a very deep result of Raynaud \cite{Ra83}:

\begin{proposition}\label{qu2-2}
Let $T$ be an $m$-dimensional abelian variety defined over $k = \overline{\Q}$, or more generally over any algebraically closed field $k$ that is not an algebraic closure of a finite field. Assume that $m \ge 1$. Then, there is a non-torsion point $P \in T(\overline{\Q})$ (resp. $P \in T(\overline{\F_q(C)})$).
\end{proposition}

The following two facts should be also well known to the experts:

\begin{proposition}\label{qu2-3}
Let $T$ be an $m$-dimensional abelian variety defined over $k = \overline{\F}_p$. Then, there is no non-torsion point $P \in T(\overline{\F}_p)$.
\end{proposition}

\begin{proof} Proof here is also kindly informed by Professor H\'el\`ene Esnault. Any variety V of finite type defined over $\overline{\F}_p$ is defined over a finite field $\F_q$ for some $q = p^m$, thus $V(\F_{q^s})$ is finite 
for all $s \ge 1$. So if $V = T$ is an abelian variety, then $T(\F_{q^s})$ is a finite group. As $P \in T(\F_{q^s})$ for some $s$, $P$ is a torsion point. 
\end{proof}

\begin{proposition}\label{prop2-2}
Let $T$ be a non-projective simple complex torus. Then $T$ admits no non-trivial meromorphic fibration. In particular, $a(T) = 0$ and any bimeromorphic automorphism of $T$, which is necessarily biholomorphic, is primitive.
\end{proposition}

\begin{proof} Here everything is in the complex analytic category. We consider complex algebraic varieties as analytic spaces. So, we call a $\C$-valued point of a complex algebraic variety simply a point.

Assume to the contrary that there is a non-trivial fibration 
$\varphi : T \cdots\to B$. As $T$ is simple, any resolution of any subvariety of $T$, other than $T$ is of general type by \cite[Chap. IV, Sect. 10]{Ue75}. In particular, they are all algebraic, $\varphi$ is not almost regular (otherwise a resolution of a general fiber would be of Kodaira dimension $0$ by the adjunction formula) and any resolutions $S_{b, i}$ of any irreducible component of $F_{b,i}$ of any fiber $F_b$ of $\varphi$ are algebraic (as a variety of general type is algebraic by definition). 

Let $I(\varphi)$ be the indeterminacy locus of $\varphi$ and $\nu : V \to T$ be a Hironaka's resolution of $I(\varphi)$ 
and $\psi : V \to B$ be the induced morphism. Let $E_l$ be the exceptional divisors of $\nu$. Then $E_l$ is also algebraic as so is any irreducible component $I_k$ of $I(\varphi)$ is algebraic. Also, any general fibers of $\psi$ is algebraic, as they are Hironaka's resolutions of fibrs of $\varphi$ which is algebraic. As $\varphi$ is not almost regular, it follows that one of $E_l$, say $E_1$, dominates $B$. 

Recall that any two points of complex algebraic variety are connected by a finite chain of algebraic curves. Thus, any two general points of $V$ are connected by a finite chain of algebraic curves, via fibers and $E_1$. Thus $V$ is algebraic by a result of Campana \cite[Page 212, Cor]{Ca81}. Hence so is $T$, as $T$ is bimeromorphic to $V$. As $T$ is compact K\"ahler manifold, it is then projective by a famous result of Moishezon, a contradiction to our assumption that $T$ is not projective. 
This completes the proof. 
\end{proof}

\section{Proof of Theorem \ref{thm1} and Corollary \ref{cor1}.}

Let us first prove Theorem \ref{thm1}. The statement (2) is a special case of Theorem \ref{prop2-1}. Let us show the statement (1). 

By Theorem \ref{thm0} and Corollary \ref{cor2}, we may and will 
assume that $S$ is either 

(i) a smooth rational surface; 

(ii) a K3 surface; 

(iii) an Enriques surface; or

(vi) an abelian surface.

Again by Theorem \ref{thm0}, we may and will further assume that $g \in {\rm Aut}\, (S)$ and $(g^*)^n = id$ on ${\rm NS}\, (S)$ as well as on $H_{{\rm et}}^k(S, \Q_l)$. Here $l$ is a prime number such that $l \not= p$. 

Consider the cases (i), (ii), and (iii) first.

Then ${\rm Pic}\, (S) \simeq {\rm NS}\, (S)$ except for (iii) with $p=2$, and  ${\rm Pic}\, (S)/{\rm Pic}^{\tau}\, (S) \simeq {\rm NS}\, (S)$ for (iii)  with $p=2$. Here ${\rm Pic}^{\tau}\, (S)$ is a finite group scheme of length $2$ (\cite{BM77}). Thus $g^*$ as an automorphism of ${\rm Pic}\, (S)$ is also of finite order. Hence $g$ is not primitive by Lemma \ref{lem1} (2). 

Consider the case (iv). Recall the following fact. Here for an abelian variety $T$, we denote by ${\rm Aut}_{{\rm group}}\, (T)$, the group of automorphisms of $T$ as algebraic 
group, and by $T(k) = \{t_P | P \in T(k)\}$, the group of translation automorphisms of $T$. 
\begin{proposition}\label{prop2-4}
Let $T$ be an abelian variety. Then:

(1) ${\rm Aut}_{{\rm variety}}\, (T)$ is the semi-direct product of 
${\rm Aut}_{{\rm group}}\, (T)$ and $T(k) = \{t_P | P \in T(k)\}$. In the semi-direct product, $T(k)$ is a normal subgroup. 

(2) The representation of ${\rm Aut}_{\rm group}\, (T)$ on $H_{{\rm et}}^1(T, \Q_l)$ is faithful and the representation of $T(k)$ on $H_{{\rm et}}^1(T, \Q_l)$ is trivial.  
\end{proposition}

\begin{proof} The assertion (1) is proved in \cite[p. 43, Cor. 1]{Mu70} and the first assertion of (2) is proved, in a slightly different terms, in \cite[p. 176. Thm. 3]{Mu70}. As $T$ is irreducible, $T(k) \subset {\rm Aut}^0(T)$. This implies the second assertion of (2). 
\end{proof}

Let us return back to the case (iV). By Proposition \ref{prop2-4} (1), one can write 
$$g = t_P \circ h\,\, $$
where $h \in {\rm Aut}_{{\rm group}}\, (S)$ and $P \in S(k)$.

We show that $h = id_S$. Consider the endomorphism
$$f := id_S - h : S \to S\,\, .$$
If $f$ would not be surjective but not $0$, then the identity component $E$ of ${\rm Ker}\, f$ would be a one dimensional closed algebraic subgroup of $S$, i.e., an elliptic curve. As $h(E) = E$, the quotient map $\pi : S \to S/E$ would be preserved by $g = t_P \circ h$, a contradiction to the fact that $g$ 
is primitive. 

If $f$ would be surjective, then there would be $Q \in S(k)$ such that $P = id_S(Q) - h(Q)$. Then $g(Q) = Q$ and $g^n(Q) = Q$. 

On the other hand, as $g^n \in {\rm Aut}^0(S)$, the automorphism $g^n$ acts on $H_{{\rm et}}^1(S, \Q_l)$ as identity. Thus $g^n$ is a translation automorphism of $S$ by Proposition \ref{prop2-4} (2). Then $g^n = id$ by $g^n(Q) = Q$, and therefore $g$ is of finite order, a contradiction to the fact that $g$ is of infinite order if $g$ is primitive (Lemma \ref{lem1} (3)). 

Hence $f = 0$, that is, $h = id_S$. Thus $g = t_P$. Then, by Proposition \ref{lem2-2}, $S$ is simple. Also, by Lemma \ref{lem1}(3), $P$ is not a torsion point of $S$. 

This completes the proof of Theorem \ref{thm1}.

Now we prove Corollary \ref{cor1}. 

Let us first show Corollary \ref{cor1} (1). As mentioned in Introduction, there is a simple abelian surface $T$ defined over any algebraically closed field $k$ (\cite{Mo77}, \cite{HZ02}). We apply this for $k = \overline{\Q}$ (resp. $k = \overline{\F_q(C)}$). Then $T(\overline{\Q})$ (resp. $\overline{\F_q(C)}$) has a non-torsion point $P$ by Proposition \ref{qu2-2}. Then $(T, t_P)$ saitisfies the requirement of Theorem \ref{thm1} (2). This implies Corollary \ref{cor1} (1). 

Let us next show Corollary \ref{cor1} (2). If there would exist such a pair $(S, g)$, then $(S, g)$ would be birationally conjugate to $(A, t_P)$ in Theorem \ref{thm1} (1). However, by Proposition \ref{qu2-3}, $t_P$ is of finite order, a contradiction to the fact that $t_P$ is of infinite order in Theorem \ref{thm1} (1). This implies Corollary \ref{cor1} (2). 

This completes the proof of Corollary \ref{cor1}.

\end{document}